\documentclass[12pt]{amsart}
\usepackage{amssymb, colordvi}
\usepackage{multirow}
\usepackage{graphicx}

\newtheorem{theorem}{Theorem}

\newtheorem{lemma}[theorem]{Lemma}

\newtheorem{ex}[theorem]{Example}
\newtheorem{rem}[theorem]{Remark}

\newenvironment{remark}{\begin{rem}\rm}{\end{rem}}
\newcounter{FNC}[page]
\def\fauxfootnote#1{{\addtocounter{FNC}{2}$^\fnsymbol{FNC}$%
     \let\thefootnote\relax\footnotetext{$^\fnsymbol{FNC}$#1}}}
\newcommand{\calH}{\mathcal{H}}
\newcommand{\calM}{\mathcal{M}}
\newcommand{\calV}{\mathcal{V}}
\newcommand{\calW}{\mathcal{W}}

\newcommand{\by}{\mbox{{\bf y}}}
\newcommand{\bz}{\mbox{{\bf z}}}

\renewcommand{\P}{\mathbb{P}}
\newcommand{\R}{\mathbb{R}}
\newcommand{\Z}{\mathbb{Z}}
\newcommand{\C}{\mathbb{C}}

\newcommand{\Jac}{\mbox{\rm Jac}}
\newcommand{\ubc}{\mbox{\rm ubc}}

\title[Fewnomial bounds for completely mixed polynomial systems]{Fewnomial bounds for
  completely mixed polynomial systems}  

\author{Fr\'ed\'eric Bihan}
\address{Laboratoire de Math\'ematiques\\
         Universit\'e de Savoie\\
         73376 Le Bourget-du-Lac Cedex\\
         France}

\email{Frederic.Bihan@univ-savoie.fr}
\urladdr{http://www.lama.univ-savoie.fr/\~{}bihan/}

\author{Frank Sottile}
\address{Department of Mathematics\\
         Texas A\&M University\\
         College Station\\
         Texas \ 77843\\
         USA}
\email{sottile@math.tamu.edu}
\urladdr{http://www.math.tamu.edu/\~{}sottile/}
\thanks{Sottile supported by NSF  grant DMS-0701050 and Texas A\&M ITRAG}  


\begin{document}
%%%%%%%%%%%%%%%%%%%%%%%%%%%%%%%%%%%%%%%%%%%%%%%%%%%%%%%%%%%%%%%%%%%%%%%%%%%

%%%%%%%%%%%%%%%%%%%%%%%%%%%%%%%%%%%%%%%%%%%%%%%%%%%%%%%%%%%%%%%%%%%%%%%%%%%
\begin{abstract}
 We give a bound for the number of real solutions to systems of $n$ polynomials in $n$
 variables, where the monomials appearing in different polynomials are distinct.
 This bound is smaller than the fewnomial bound if this structure of
 the polynomials is not taken into account.
\end{abstract}
%%%%%%%%%%%%%%%%%%%%%%%%%%%%%%%%%%%%%%%%%%%%%%%%%%%%%%%%%%%%%%%%%%%%%%%%%%%
\maketitle

%%%%%%%%%%%%%%%%%%%%%%%%%%%%%%%%%%%%%%%%%%%%%%%%%%%%%%%%%%%%%%%%%%%%%%%%%%%
%
\section*{Introduction}
%
%%%%%%%%%%%%%%%%%%%%%%%%%%%%%%%%%%%%%%%%%%%%%%%%%%%%%%%%%%%%%%%%%%%%%%%%%%%
In 1980, A.~Khovanskii~\cite{Kh91} showed that a system of $n$ polynomials in $n$ variables
involving $l{+}n{+}1$ distinct monomials has less than 
 \begin{equation}\label{Eq:Kh_bond}
    2^{\binom{l+n}{2}} (n+1)^{l+n}
 \end{equation}
non-degenerate positive solutions.
This fundamental result established the principle  
that the number of real solutions to such a system should have an upper bound that depends
only upon its number of terms.
Such results go back to Descartes~\cite{D1637}, whose rule of signs
implies that a univariate polynomial having $l{+}1$ terms has at most
$l$ positive zeroes.
This principle was formulated by Kushnirenko, who coined the term
``fewnomial'' that has come to describe results of this type.

Khovanskii's bound~\eqref{Eq:Kh_bond} is the specialization to
polynomials of his bound for a more general class of functions.
Recently, the significantly lower bound of
 \begin{equation}\label{Eq:pos_bound}
   \frac{e^2+3}{4} 2^{\binom{l}{2}}n^l
 \end{equation}
was shown~\cite{BS07} for polynomial fewnomial systems.
This took advantage of some geometry specific to polynomial systems, but was otherwise
based on Khovanskii's methods.
The significance of this bound is that it is sharp in the sense that for fixed $l$ there
are systems with $O(n^l)$ positive solutions~\cite{BRS}.
Modifying the proof~\cite{BBS} leads to the bound
 \begin{equation}\label{Eq:Real_bound}
   \frac{e^{\Red{4}}+3}{4} 2^{\binom{l}{2}}n^l
 \end{equation}
for the number of real solutions, when the exponent vectors of the monomials generate
the integer lattice---this condition disallows trivial solutions that differ from
other solutions only by some predictable signs.

These bounds hold in particular if each of the polynomials involve the
same $1{+}l{+}n$ monomials, which is referred to as an unmixed polynomial system.
By Kushnirenko's principle, we should expect a lower bound if not all monomials appear in
every polynomial.

Such an approach to fewnomial bounds, where we take into account differing structures of
the polynomials, was in fact the source of the first result in this subject.
In 1978, Sevostyanov proved there is a function $N(d,m)$ such that if
the polynomial $f(x,y)$ has degree $d$ and the polynomial $g(x,y)$ has $m$ terms, then
the system 
 \begin{equation}\label{eq:bivariate_system}
   f(x,y)\ =\ g(x,y)\ =\ 0
 \end{equation}
has at most $N(d,m)$ non-degenerate positive solutions.
This result has unfortunately never been published\fauxfootnote{A description of this and
  much more is found in Anatoli Kushnirenko's letter to Sottile~\cite{Kush_note}.}.
A special case was recently refined by Avenda\~{n}o~\cite{Av}, who showed that if
$f$ is linear, then~\eqref{eq:bivariate_system} has at most $6m{-}4$ real solutions.

Li, Rojas, and Wang~\cite{LRW03} showed that a fewnomial
system~\eqref{eq:bivariate_system} where $f$ has 3 terms will have at most $2^m-2$
positive solutions (when $m=3$, the bound is lowered to 5).
More generally, they showed that the number of positive solutions to a system 
 \begin{equation}\label{Eq:nxn_system}
    g_1(x_1,\dotsc,x_n)\ =\ 
    g_2(x_1,\dotsc,x_n)\ =\ \dotsb\ =\ 
    g_n(x_1,\dotsc,x_n)\ =\ 0
 \end{equation}
is at most $n+n^2+\dotsb+n^{m-1}$, when each of $g_1,\dotsc,g_{\Red{n-1}}$ is
a trinomial and $g_n$ has $m$ terms.
These bounds are significantly smaller than the corresponding bounds of~\cite{BS07},
which are $\frac{e^2+3}{4}2^{\binom{m-1}{2}}n^{m-1}$ in both cases.
Their methods require that at most one polynomial is not a
trinomial and apparently do not generalize.
However, their results show that the fewnomial bound can be improved when the polynomials
have additional structure.\smallskip

We take the first steps towards improving the fewnomial bounds~\eqref{Eq:pos_bound}
and~\eqref{Eq:Real_bound} when the polynomials have additional structure, but no limit on
their numbers of monomials.
That is, if the polynomial $g_i$ in~\eqref{Eq:nxn_system} has $2+l_i$ terms with $l_i>0$,
we seek bounds on the number of non-degenerate positive solutions 
that are smaller in order than $2^{\binom{l}{2}}n^l$, where
$l{+}n{+}1$ is the total number of terms in all polynomials.
Note that $l\leq l_1+\dotsb+l_n$.
The reason for our choice of parameterization of these systems is that if
some $l_i=0$, there is a change of variables which reduces the number
of variables, eliminates $g_i$ from the list polynomials, and does not change the number
of monomials in the other polynomials, nor the number of positive solutions.

%%%%%%%%%%%%%%%%%%%%%%%%%%%%%%%%%%%%%%%%%%%%%%%%%%%%%%%%%%%%%%%%%%%%%%%%%%%%
\begin{theorem}\label{T:one}
  Suppose that each polynomial $g_i$ in$~\eqref{Eq:nxn_system}$ has a constant term, but
  otherwise all monomials are distinct, so that the system involves $l{+}n{+}1$ monomials where  
  $l=l_1+\dotsb+l_n$.
  Then the number of non-degenerate non-trivial non-zero real solutions$~\eqref{Eq:nxn_system}$
  is at most 
\[
   \tfrac{e^4+3}{4} \cdot 2^{\binom{l}{2}} \tbinom{l}{l_1, \dotsc,l_n}\ ,
\]
  and the number of those which are positive is at most 
\[
   \tfrac{e^2+3}{4} \cdot 2^{\binom{l}{2}} \tbinom{l}{l_1, \dotsc,l_n}\ .
\]
\end{theorem}
%%%%%%%%%%%%%%%%%%%%%%%%%%%%%%%%%%%%%%%%%%%%%%%%%%%%%%%%%%%%%%%%%%%%%%%%%%%%

The bounds of Theorem~\ref{T:one} are strictly smaller than those of~\cite{BBS,BS07}, for
\[
   n^l\ =\ \sum \tbinom{l}{l_1, \dotsc,l_n}\,,
\]
the sum over all $0\leq l_i$ with $l_1+\dotsb+l_n=l$.

In Theorem~\ref{T:one} the bound for positive solutions holds if we allow real-number
exponents, and the first bound is for all non-zero real solutions when the exponents of
the monomials span a subgroup of $\Z^n$ of odd index 
(for otherwise there are trivial solutions).
We only need to prove this for $n\geq 2$, as these bounds exceed Descartes' bound when $n=1$.

We establish Theorem~\ref{T:one} by modifying the arguments of~\cite{BBS,BS07}.
In particular, we apply a version of Gale duality~\cite{BS_Gale} to replace the system of
polynomials by a system of master functions in the complement of a hyperplane arrangement in
$\R^l$, and then estimate the number of solutions by repeated applications of the
Khovanskii-Rolle Theorem applied to successive Jacobians of the system of master
functions. 
This modification is not as straightforward as we have just made it sound.
First, the arguments we modify require that the hyperplane arrangement be in general
position in $\R\P^l$, but in the case here, the hyperplanes are
arrangements of certain normal crossings divisors in the product of
projective spaces $\R\P^{l_1}\times\dotsb\times\R\P^{l_n}$. 
We exploit the special structure of chambers in this complement, together with the
multihomogenity of the Jacobians to obtain the smaller bounds of Theorem~\ref{T:one}.

A more fundamental yet very subtle modification in the arguments is that they require
certain successive Jacobians to meet transversally.
While this can be arranged in~\cite{BBS,BS07} by varying the parameters, 
we do not have such freedom here and the Jacobians
(once $l>2$) can meet non-transversally, and in fact non-properly
when $l>3$.
Thus we cannot simply apply the Khovanskii-Rolle Theorem, but must
provide a modification in
the arguments.

%%%%%%%%%%%%%%%%%%%%%%%%%%%%%%%%%%%%%%%%%%%%%%%%%%%%%%%%%%%%%%%%%%%%%%%%%%%%
%
%
\section{Gale duality for completely mixed polynomial systems}
We do not prove Theorem~\ref{T:one} by arguing directly on the polynomial system, but
rather on a different, equivalent Gale-dual system defined in the complement of a
normal-crossings divisor in the product of projective spaces 
$\R\P^{l_1}\times\dotsb\times\R\P^{l_n}$.

An integer vector $w\in\Z^n$ may be regarded as the exponent of a
Laurent monomial
\[
   \Blue{x^w}\ :=\ x_1^{w_1}x_2^{w_2}\dotsb x_n^{w_n}\,.
\]
Given a collection $\calW\subset\Z^n$ of exponent vectors and 
coefficients $\{a_w\in \C\}$, we obtain the Laurent polynomial
\[
     g(x)\ =\ \sum_{w\in\calW} a_w x^w\ .
\]
This is naturally defined on the complex torus $(\C^\times)^n$ or the real torus $(\R^\times)^n$.
If we restrict the variable $x$ to have positive real components
($x\in\R^n_>$), then we may allow the exponents $w$ to have real-number components.

Fix positive integers $n,l_1,\dotsc,l_n$ with $n>1$ and set
$\Blue{l}:=l_1+\dotsb+l_n$.
We consider systems of Laurent polynomials with real coefficients of the form
 \begin{equation}\label{Eq:CM_system}
   g_1(x_1,\dotsc,x_n)\ =\ 
   g_2(x_1,\dotsc,x_n)\ =\ \dotsb\ =\ 
   g_n(x_1,\dotsc,x_n)\ =\ 0\,,
 \end{equation}
where each polynomial $g_i$ has $l_i+2$ monomials, one of which is a
constant term, and there are no other monomials common to any pair of
polynomials.
The condition that each polynomial has a constant term may be arranged by
multiplying it by a suitable monomial.
This transformation does not change the solutions to the system~\eqref{Eq:CM_system}.

In this case, the system~\eqref{Eq:CM_system} has $l{+}n{+}1$
monomials, so it has at most $\frac{e^2+3}{4}2^{\binom{l}{2}}n^l$ positive solutions.
If the exponents of the monomials span a sublattice of odd index in
$\Z^n$, then the system has at most $\frac{e^4+3}{4}2^{\binom{l}{2}}n^l$ non-zero
real solutions.

Here, we prove Theorem~\ref{T:one}, which improves these bounds for
the system~\eqref{Eq:CM_system} by taking into account the special
structure of the polynomials $g_i$.
This follows the proofs of the bounds in~\cite{BS07,BBS}, but with
several essential and subtle modifications.

%%%%%%%%%%%%%%%%%%%%%%%%%%%%%%%%%%%%%%%%%%%%%%%%%%%%%%%%%%%%%%%%%%%%%%%%%%%%
%
%
\subsection{Reduction to Gale dual system}

For $i=1,\dotsc,n$, let $\{0,w_{i,0},w_{i,1},\dotsc,w_{i,l_i}\}$ be
the exponents of monomials in the polynomial $g_i$, and
rewrite the equation $g_i=0$ as
 \begin{eqnarray*}
   x^{w_{i,0}} &=& a_{i,0} + a_{i,1}x^{w_{i,1}} + \dotsb + a_{i,l_i}x^{w_{i,l_i}}\\ 
             &=& \Blue{p_i}(x^{w_{i,1}},\dotsc,x^{w_{i,l_i}})\,,
 \end{eqnarray*}
where $p_i$ is a degree 1 polynomial in its arguments.

A linear relation among the exponent vectors,
\[
   \sum_{i=1}^{n} \bigl( \alpha_{i,0}w_{i,0} + \alpha_{i,1}w_{i,1} + 
      \dotsb +\alpha_{i,l_i}w_{i,l_i}\bigr) \ =\ 0\,
\]
corresponds to the identity
\[
   \prod_{i=1}^n \Bigl ((x^{w_{i,0}})^{\alpha_{i,0}} \cdot
      \prod_{j=1}^{l_i}  (x^{w_{i,j}})^{\alpha_{i,j}} \Bigr)\ =\ 1\,.
\]
Substituting $x^{w_{i,0}}=p_i(x^{w_{i,1}},\dotsc,x^{w_{i,l_i}})$ into this, we
obtain the consequence of~\eqref{Eq:CM_system},
 \begin{equation}\label{Eq:consequence}
   \prod_{i=1}^n \Bigl (p_i(x^{w_{i,1}},\dotsc,x^{w_{i,l_i}})^{\alpha_{i,0}} \cdot
      \prod_{j=1}^{l_i}  (x^{w_{i,j}})^{\alpha_{i,j}} \Bigr)\ =\ 1\,.
 \end{equation}

Let $\alpha^{(1)},\dotsc,\alpha^{(l)}\in\Z^{n+l}$ be a basis for the subgroup of integer
linear relations among the exponent vectors $w_{i,j}\in\Z^n$, which is saturated.
This gives $l$ independent equations of the form~\eqref{Eq:consequence},
one for each relation $\alpha^{(k)}$.
Under the substitution $y_{i,j}=x^{w_{i,j}}$ for $i=1,\dotsc,n$ and
$j=\Red{1},\dotsc,l_i$,
we obtain the \Blue{{\sl Gale dual system}},
 \begin{equation}\label{Eq:GD}
   \prod_{i=1}^n \Bigl (p_i(y_{i,1},\dotsc,y_{i,l_i})^{\alpha^{(k)}_{i,0}} \cdot
      \prod_{j=1}^{l_i}  (y_{i,j})^{\alpha^{(k)}_{i,j}} \Bigr)\ =\ 1\ \quad
      \mbox{\rm for}\quad k=1,\dotsc,l\,,    
 \end{equation}
which is a consequence of~\eqref{Eq:CM_system} and is valid where $y_{i,j}\neq 0$ and 
$p_i(y_{i,1},\dotsc,y_{i,l_i})\neq 0$.

%%%%%%%%%%%%%%%%%%%%%%%%%%%%%%%%%%%%%%%%%%%%%%%%%%%%%%%%%%%%%%%%%%%%%%%%%%%%
\begin{theorem}[Gale duality for polynomial systems~\cite{BS_Gale}]\label{Th:Gale}
 Suppose that the exponent vectors $w_{i,j}$ span $\Z^n$, and that
 one of the systems~\eqref{Eq:CM_system} or~\eqref{Eq:GD} is a complete intersection. 
 Then the map $x\mapsto y$ defined by
\[
   y_{i,j}\ =\ x^{w_{i,j}}\ \quad
   \mbox{for}\ i=1,\dotsc,n\ \mbox{and}\ j=1,\dotsc,l_i\,,
\]
 gives a scheme-theoretic isomorphism between the solutions to~\eqref{Eq:CM_system} in
 $(\C^\times)^n$ and solutions to~\eqref{Eq:GD} in
\[
   \{y\in(\C^\times)^l\mid p_i(y_{i,1},\dotsc,y_{i,l_i})\neq 0\ \mbox{ for}\ i=1,\dotsc,n\}\,.
\]
 If the exponent vectors span a sublattice of odd index, then this restricts to an
 isomorphism between the corresponding real analytic schemes of solutions.

 If we further relax the conditions on the exponents, allowing them to be real vectors
 which span $\R^n$, then this becomes 
 an isomorphism of real analytic schemes between positive solutions
 of~\eqref{Eq:CM_system} and solutions of~\eqref{Eq:GD} in the positive chamber
\[
   \Blue{\Delta_+}\ :=\ 
    \{ y \in \R^l_>\mid p_i(y_{i,1},\dotsc,y_{i,l_i})>0\ \mbox{ for}\ i=1,\dotsc,n\}\,.
\]
\end{theorem}
%%%%%%%%%%%%%%%%%%%%%%%%%%%%%%%%%%%%%%%%%%%%%%%%%%%%%%%%%%%%%%%%%%%%%%%%%%%%

%%%%%%%%%%%%%%%%%%%%%%%%%%%%%%%%%%%%%%%%%%%%%%%%%%%%%%%%%%%%%%%%%%%%%%%%%%%%
\begin{remark}
 The proof realizes both systems as the same intersection in $\C\P^{n+l}$ between
 an $n$-dimensional toric variety (corresponding to the exponents of the polynomials $g_i$)
 and an $l$-dimensional linear space (corresponding to the coefficients of the
 $g_i$).
 More specifically, to their points of intersection off the coordinate planes.
 This identification restricts to the points in $\R\P^{n+l}$, and also to points in the
 positive orthant of $\R\P^{n+l}$.
\end{remark}
%%%%%%%%%%%%%%%%%%%%%%%%%%%%%%%%%%%%%%%%%%%%%%%%%%%%%%%%%%%%%%%%%%%%%%%%%%%%

Askold Khovanskii has pointed out that the bounds of~\cite{BBS,BS07} may be established by
working directly on the intersection of the toric variety with the linear space in the
complement of the coordinate planes in $\R\P^{n+l}$, and then using his general method of
bounds for separating solutions of Pfaff equations~\cite[Ch.~3]{Kh91}.
Thus they are a consequence of his general theorem that there exists some
bound. 
Nevertheless, the bounds of~\cite{BBS,BS07} are significant in that they are sharp for $l$
fixed and $n$ large, and that the bound in~\cite{BBS} is for all real solutions, yet is
not much larger than the bound for positive solutions.

Here, we shall also use the formulation as Gale dual systems.
This is because the linear space does not meet the 
coordinate planes in a divisor with normal crossings, due to the special form of the
polynomials $g_i$.
This technical assumption is necessary to obtain good bounds from 
Khovanskii's method in these cases.

Rather than use the pullback of the coordinate hyperplanes in $\R\P^l$, we work instead
with hypersurfaces in the product $\R\P^{l_1}\times\dotsb\times\R\P^{l_n}$
which come from the coordinate hyperplanes and the hyperplane $p_i=0$ in each
factor, and which have normal crossings.
This is further justified, as our arguments for Theorem~\ref{T:one} exploit a block
structure in the variables corresponding to the factors of this product of projective
spaces. 

%%%%%%%%%%%%%%%%%%%%%%%%%%%%%%%%%%%%%%%%%%%%%%%%%%%%%%%%%%%%%%%%%%%%%%%%%%%%
%
%
\section{Proof of Theorem~\ref{T:one}}

Let $n,l_1,\dotsc,l_n$ be positive integers with $n>1$ and set $l:=l_1+\dotsb+l_n$.
For each $i=1,\dotsc,n$ let $\Blue{\bz_i}:=(z_{i,1},\dotsc,z_{i,l_i})$ 
be a collection of $l_i$ real variables and set
 \[
   \Blue{q_i(\bz_i)}\ :=\ 1+z_{i,1}+z_{i,2}+\dotsb+z_{i,l_i}\,.
 \]
Let $\Blue{\calH_i}\subset\R\P^{\ell_i}$ be the arrangement of $l_i+2$ hyperplanes 
consisting of the coordinate hyperplanes and
the hyperplane $q_i(\bz_i)=0$.
Write $\Blue{\calM_i}\subset\R^{l_i}$ for the complement of $\calH_i$.

Then $\bz:=(\bz_1,\dotsc,\bz_n)$ are $l$ real variables.
Let $b_1,\dotsc,b_n\in\R^\times$ be non-zero real numbers and
$\alpha^{(1)},\dotsc,\alpha^{(l)}$ be independent vectors in $\R^{n+l}$.
For each $k=1,\dotsc,l$, set
\[
   \Blue{f_k(\bz)}\ :=\ \prod_{i=1}^n\Bigr( |q_i(\bz_i)|^{\alpha^{(k)}_{i,0}} \cdot
         \prod_{j=1}^{l_i} |z_{i,j}|^{\alpha^{(k)}_{i,j}}\Bigl)
   \qquad\mbox{and}\qquad
   \Blue{d_k}\ :=\ \Bigl(\prod_{i=1}^n |b_i|^{\alpha_{i,0}^{(k)}}\Bigr)^{-1}\,,
\]
and let $\Blue{g_k(\bz)}:=f_k(\bz)-d_k$.
Write $\R\P$ for the product $\R\P^{l_1}\times\dotsb\times\R\P^{l_n}$ and let
$\Blue{\calM}:=\calM_1\times\dotsb\times\calM_n$.
This is the complement of $l+2n$ hypersurfaces in 
$\R\P$ that meet with normal crossings.
Write $\Blue{\calH}$ for this arrangement of hypersurfaces, which is
\[
   \bigcup_{i=1}^n \R\P^{l_1}\times\dotsb\times\R\P^{l_{i-1}}\times\,\calH_i\,\times
           \R\P^{l_{i+1}}\times\dotsb\times\R\P^{l_n}\ .
\]
These hypersurfaces stratify $\R\P$ with 
the $l$-dimensional strata the connected components of $\calM$, which we will call the
\Blue{{\sl chambers}} of $\calH$.
A non-empty intersection of $k$ of the hypersurfaces is smooth of codimension $k$, is
isomorphic to a product of projective spaces and is itself stratified by its intersection with
the other hypersurfaces. 
The chambers of this stratification are the $l{-}k$-dimensional \Blue{{\sl faces}} of $\calH$.

%%%%%%%%%%%%%%%%%%%%%%%%%%%%%%%%%%%%%%%%%%%%%%%%%%%%%%%%%%%%%%%%%%%%%%%%%%%%
\begin{theorem}\label{T:reduction}
  The system
 \begin{equation}\label{Eq:reduction}
   g_1(\bz)\ =\ g_2(\bz)\ =\ \dotsb\ =\ g_l(\bz)\ =\ 0\
 \end{equation}
 has at most
\[
  \tfrac{e^4+3}{4}\cdot 2^{\binom{l}{2}} \tbinom{l}{l_1, \dotsc,l_n}\ ,
\]
 non-degenerate solutions in  $\calM$, and at most
\[
  \tfrac{e^2+3}{4}\cdot 2^{\binom{l}{2}} \tbinom{l}{l_1, \dotsc,l_n}\ .
\]
 non-degenerate solutions in any connected component of $\calM$.
\end{theorem}
%%%%%%%%%%%%%%%%%%%%%%%%%%%%%%%%%%%%%%%%%%%%%%%%%%%%%%%%%%%%%%%%%%%%%%%%%%%%

%%%%%%%%%%%%%%%%%%%%%%%%%%%%%%%%%%%%%%%%%%%%%%%%%%%%%%%%%%%%%%%%%%%%%%%%
\begin{proof}[Proof of Theorem~$\ref{T:one}$]
By Theorem~\ref{Th:Gale}, it suffices to consider 
an equivalent Gale dual system~\eqref{Eq:GD}.
Write $p_i(\by_i)$ for $p_i(y_{i,1},\dotsc,y_{i,l_i})$.
We bound the solutions to 
 \begin{equation}\label{Eq:GD_2}
   1\ =\ \prod_{i=1}^n\Bigr( p_i(\by_i)^{\alpha^{(k)}_{i,0}} \cdot
         \prod_{j=1}^{l_i} y_{i,j}^{\alpha^{(k)}_{i,j}}\Bigl)\qquad  \mbox{for}\ k=1,\dotsc,l\,,
 \end{equation}
that (i) are real and also those (ii) that lie in the positive chamber
\[ 
    \Blue{\Delta_+}\ :=\ \{\by\mid y_{i,j}>0\quad\mbox{and}\quad p_i(\by_i)>0,
     \quad\mbox{for all }i,j\}\,.
\]

The system~\eqref{Eq:GD_2} is a subsystem of the system
 \begin{equation}\label{Eq:GD_3}
   1\ =\ \prod_{i=1}^n\Bigr( |p_i(\by_i)|^{\alpha^{(k)}_{i,0}} \cdot
         \prod_{j=1}^{l_i} |y_{i,j}|^{\alpha^{(k)}_{i,j}}\Bigl)\qquad \mbox{for}\ k=1,\dotsc,l\,.
 \end{equation}
This has the same solutions as~\eqref{Eq:GD_2} in the positive chamber $\Delta_+$.
It is the disjunction of systems Gale dual to the systems 
 \[
    x^{w_{i,0}}\ =\ \pm a_{i,0}\,\pm\, a_{i,1}x^{w_{i,1}}\,\pm\ \dotsb\
    \pm\,a_{i,l_i}x^{w_{i,l_i}}
    \qquad\mbox{for}\ i=1,\dotsc,n\,,
 \]
as $\pm$ ranges over all sign choices, and so its solutions include the real
solutions to~\eqref{Eq:GD_2}. 
Since there are finitely many such systems, we may assume that they are simultaneously
non-degenerate. 

Replacing each variable $y_{i,j}$ by $a_{i,0}z_{i,j}/a_{i,j}$, where $z_{i,j}$ are new
real variables, we have 
 \[
     p_i(\by_i)\ =\ a_{i,0} ( 1+ z_{i,1}+\dotsb+z_{i,l_i})\ = 
      a_{i,0} q_i(\bz_i)\,,
 \]
where $\bz_i:=(z_{i,1},\dotsc,z_{i,l_i})$.
Under this transformation, the system~\eqref{Eq:GD_3} becomes
\[
   d_k^{-1}f_k(\bz)\ =\ 1\quad(\mbox{or}\quad g_k(\bz)\ =\ 0)
   \qquad\mbox{for}\qquad k=1,\dotsc,l\,,
\]
which is just the system~\eqref{Eq:reduction}, where $b_i=a_{i,0}$.
We complete the proof of Theorem~\ref{T:one} by noting that the transformation
$\by\mapsto\bz$ transforms the domain of the Gale system into $\calM$, mapping the positive
chamber $\Delta_+$ to some chamber of $\calM$.
\end{proof}
%%%%%%%%%%%%%%%%%%%%%%%%%%%%%%%%%%%%%%%%%%%%%%%%%%%%%%%%%%%%%%%%%%%%%%%%

%%%%%%%%%%%%%%%%%%%%%%%%%%%%%%%%%%%%%%%%%%%%%%%%%%%%%%%%%%%%%%%%%%%%%%%%
\begin{remark}
 It suffices to prove Theorem~\ref{T:reduction} when the constants 
 $b=(b_1,\dotsc,b_n)$ and the exponents 
 $\alpha=(\alpha^{(1)},\dotsc,\alpha^{(l)})$ are general.
 In particular, we will assume that every submatrix of the matrix whose rows are the
 exponent vectors has full rank, and further that the constants $b$ and the exponents 
 $\alpha^{(k)}_{i,0}$ are general.
 This is sufficient because a perturbation of the system~\eqref{Eq:reduction} will not reduce
 its number of non-degenerate solutions in $\calM$.
\end{remark}
%%%%%%%%%%%%%%%%%%%%%%%%%%%%%%%%%%%%%%%%%%%%%%%%%%%%%%%%%%%%%%%%%%%%%%%%

We reduce the proof of Theorem~\ref{T:reduction} to a series of lemmas, which are proven in
subsequent sections.
For each $k=1,\dotsc,l$ set $\Blue{\phi_k(\bz)}:=\log f_k(\bz)$, which is
\[
   \phi_k(\bz)\ =\ \sum_{i=1}^n \Bigl( \alpha_{i,0}^{(k)} \log |q_i(\bz_i)|
    \ +\ \sum_{j=1}^{l_i} \alpha_{i,j}^{(k)}\log|z_{i,j}|\Bigr)\ .
\]
Then the system~\eqref{Eq:reduction} becomes
$\phi_k(\bz)=\log(d_k)$ for $k=1,\dotsc,l$.
We also consider subsets $\mu_k$ of $\calM$ defined by
 \begin{eqnarray*}
   \Blue{\mu_k}& :=&
   \{\bz\in\calM\mid \phi_m(\bz)=\log(d_m)\ \mbox{for}\ m=1,\dotsc, k{-}1\}\\
   &=&
    \{\bz\in\calM\mid f_m(\bz)=d_m\ \mbox{for}\ m=1,\dotsc, k{-}1\}\ .
 \end{eqnarray*}
%

%%%%%%%%%%%%%%%%%%%%%%%%%%%%%%%%%%%%%%%%%%%%%%%%%%%%%%%%%%%%%%%%%%%%%%%%
\begin{lemma}\label{L:boundary}
 The subset $\mu_k$ of $\calM$ is smooth and has dimension $l{-}k{+}1$.
 The points in $\calH$ lying in the closure of $\mu_k$ are a union of
 $l{-}k$ dimensional faces. 
 In the neighborhood of any point in the relative interior of such a face, $\mu_k$ may have at
 most one branch in each chamber of $\calM$ adjacent to that face.
\end{lemma}
%%%%%%%%%%%%%%%%%%%%%%%%%%%%%%%%%%%%%%%%%%%%%%%%%%%%%%%%%%%%%%%%%%%%%%%%

We will prove this lemma in \S~\ref{S:smooth}, where we also explain our genericity
hypotheses.

A polynomial $F(\bz)$ has \Blue{{\sl multidegree} $d$} if, for each
$i=1,\dotsc,n$ it has degree $d$ in the block of variables $\bz_i$.
This is typically written multidegree $(d,\dotsc,d)$, but we adopt this simplified notation as
our polynomials will have the same degree in each block of variables.

A key step in our estimate is the following modification of the Khovanskii-Rolle
Theorem~\cite[pp.~42--51]{Kh91}. 
Write $\#\calV(\psi_1,\dotsc,\psi_l)$ for the number of solutions to the system 
$\psi_1=\dotsb=\psi_l=0$.
Recall that we write $g_k(\bz)$ for $f_k(\bz)-d_k$.

%%%%%%%%%%%%%%%%%%%%%%%%%%%%%%%%%%%%%%%%%%%%%%%%%%%%%%%%%%%%%%%%%%%%%%%%%%%%
\begin{theorem}\label{T:Modified_Kh_Ro}
 There exist polynomials $F_1, F_2,\dotsc,F_l$ where $F_{l-k}$ is a polynomial of
 multidegree $2^k$ with the property that
\begin{enumerate}

 \item The system
\[
    g_1\ =\ \dotsb\ =\ g_k\ \;=\;\ F_{k+1}\ =\ \dotsb\ =\ F_l\ =\ 0
\]
   has only non-degenerate solutions in $\calM_\C$, and the system
\[
    g_1\ =\ \dotsb\ =\ g_{k-1}\ \;=\;\ F_{k+1}\ =\ \dotsb\ =\ F_l\ =\ 0
\]
     ($g_k$ is omitted) defines a smooth  curve $C_k\subset\calM$.

 \item We have the estimate
 \begin{equation}\label{Eq:MKh-Ro}
    \#\calV(g_1,\dotsc,g_k,\,F_{k+1},\dotsc,F_l)\ \leq\ 
     \ubc(C_k)\ +\ 
    \#\calV(g_1,\dotsc,g_{k-1},\,F_k,F_{k+1},\dotsc,F_l)\,,
 \end{equation}
 where $\ubc(C_k)$ is the number of unbounded components of the curve $C_k$. 
\end{enumerate}
\end{theorem}
%%%%%%%%%%%%%%%%%%%%%%%%%%%%%%%%%%%%%%%%%%%%%%%%%%%%%%%%%%%%%%%%%%%%%%%%%%%%

The estimate~\eqref{Eq:MKh-Ro} leads to the
estimate for the number of solutions to~\eqref{Eq:reduction}:
 \begin{equation}\label{Eq:estimation}
  \#\calV(g_1,\dotsc,g_l)\ \leq\ \ubc(C_1)+\dotsb+\ubc(C_l)+\#\calV(F_1,\dotsc,F_l)\ .
 \end{equation}
This holds both in the full complement $\calM$, as well as in each chamber
when we interpret the quantities in~\eqref{Eq:estimation} relative to that chamber.

%%%%%%%%%%%%%%%%%%%%%%%%%%%%%%%%%%%%%%%%%%%%%%%%%%%%%%%%%%%%%%%%%%%%%%%%%%%%
\begin{lemma}\label{L:first_estimation}
  In $\calM$ we have
 \begin{enumerate}
   \item 
     ${\displaystyle  \#\calV(F_1,\dotsc,F_l)\ \leq\  2^{\binom{l}{2}}
       \tbinom{l}{l_1, \dotsc,l_n}}$, and

   \item 
     ${\displaystyle  \ubc(C_k)\ \leq\  
        \frac{1}{2}  \cdot 2^k\cdot 2^{\binom{l-k}{2}} 
      \cdot
      \sum \tbinom{l-k}{j_1, \dotsc,j_n} \cdot \prod_{i=1}^n  \tbinom{l_i+2}{j_i+2}}$,\newline
     the sum over all $j_1,\dotsc,j_n$ with $0\leq j_i\leq l_i$ for $i=1,\dotsc,n$ 
     where $j_1+\dotsb+j_n=l-k$.
 \end{enumerate}  
  If we instead estimate these quantities in a single chamber $\Delta$ of $\calM$, then 
  the estimation $(1)$ for $\#\calV(F_1,\dotsc,F_l)$ is unchanged, but that for $(2)$ is simply  
  divided by $2^k$.
\end{lemma}
%%%%%%%%%%%%%%%%%%%%%%%%%%%%%%%%%%%%%%%%%%%%%%%%%%%%%%%%%%%%%%%%%%%%%%%%%%%%

If we use these estimates in the sum~\eqref{Eq:estimation}, we obtain
 \begin{equation}\label{big_est}
   \frac{1}{2} \left[\sum_{k=1}^l 2^k\cdot 2^{\binom{l-k}{2}} \cdot
      \sum_{(j_1,\dotsc,j_n)} 
      \tbinom{l-k}{j_1, \dotsc,j_n} \cdot \prod_{i=1}^n  \tbinom{l_i+2}{j_i+2}\right]
   \quad+\quad 2^{\binom{l}{2}} \tbinom{l}{l_1, \dotsc,l_n}\ .
 \end{equation}
%

%%%%%%%%%%%%%%%%%%%%%%%%%%%%%%%%%%%%%%%%%%%%%%%%%%%%%%%%%%%%%%%%%%%%%%%%%%%%
\begin{lemma}\label{L:further estimation}
  For $l \geq 3$ the sum in brackets in~$\eqref{big_est}$ is less than
\[
   \tfrac{e^4-1}{2}\cdot 2^{\binom{l}{2}} \tbinom{l}{l_1, \dotsc,l_n}\,.
\]
 If we instead use the estimate in a single chamber $\Delta$, dividing by $2^k$ where
 appropriate, then it becomes
\[
   \tfrac{e^2-1}{2}\cdot 2^{\binom{l}{2}} \tbinom{l}{l_1, \dotsc,l_n}\,.
\]
\end{lemma}
%%%%%%%%%%%%%%%%%%%%%%%%%%%%%%%%%%%%%%%%%%%%%%%%%%%%%%%%%%%%%%%%%%%%%%%%%%%%

Theorem~\ref{T:reduction} with $l \geq 3$ now follows from Lemma~\ref{L:further estimation}.
For $l=2$, Theorem~\ref{T:reduction} is a consequence of~\cite{LRW03}, where it is proved that
a system of two trinomial equations in two variables has at most $5$ positive solutions,
and thus at most $20$ real solutions.

It is possible to further lower the estimate for $\ubc(C_k)$ in
Lemma~\ref{L:first_estimation} and the estimates in  Lemma~\ref{L:further estimation}, but
this will not significantly affect our bounds as the estimate for $\#\calV(F_1,\dotsc,F_l)$
dominates these estimates.

%%%%%%%%%%%%%%%%%%%%%%%%%%%%%%%%%%%%%%%%%%%%%%%%%%%%%%%%%%%%%%%%%%%%%%%%
We establish Lemma~\ref{L:boundary} in Section~\ref{S:smooth},
Theorem~\ref{T:Modified_Kh_Ro} in Section~\ref{S:Thm7},
Lemma~\ref{L:first_estimation} in Section~\ref{L:Main_est},
and finally Lemma~\ref{L:further estimation} in Section~\ref{S:further estimation}.

%%%%%%%%%%%%%%%%%%%%%%%%%%%%%%%%%%%%%%%%%%%%%%%%%%%%%%%%%%%%%%%%%%%%%%%%%%%%%%%%%%
%
\subsection{Proof of Lemma~\ref{L:boundary}}\label{S:smooth}

Set $F(\bz)=(f_1(\bz),f_2(\bz),\dotsc,f_{k-1}(\bz))$, where
\[
   f_m(\bz)\ =\ \prod_{i=1}^n\Bigl( |q_i(\bz_i)|^{\alpha^{(m)}_{i,0}}
    \cdot \prod_{j=1}^{l_i} |z_{i,j}|^{\alpha^{(m)}_{i,j}}\Bigr)\ .
\]
Then $\mu_k=F^{-1}(d_1,\dotsc,d_{k-1})$.
We would like to conclude that $\mu_k$ is smooth and has dimension $l{-}k{+}1$ using
Sard's Theorem.

To do that, observe that if the exponents $\alpha^{(m)}_{i,j}$ are sufficiently general
(for example, when the matrix whose rows are the vectors $\alpha^{(m)}$ for $m=1,\dotsc,k{-}1$
 has no vanishing maximal minor),
then $F$ is a $C^\infty$ map $\calM\to\R^{k-1}_>$ with dense image.
Since 
\[
   d_m^{-1}\ =\ \prod_{i=1}^n |b_i|^{\alpha^{(m)}_{i,0}}\ ,
\]
we see that choosing $b_i$ and $\alpha^{(m)}_{i,0}$ we can ensure that $(d_1,\dotsc,d_k)$ is a
regular value of the map $F$, and so by Sard's Theorem, $\mu_k$ is indeed smooth.

The second statement follows by arguments similar to the proof of Lemma~3.8 in~\cite{BS07}.
That proof requires the genericity hypothesis on the matrix of exponent vectors.

%%%%%%%%%%%%%%%%%%%%%%%%%%%%%%%%%%%%%%%%%%%%%%%%%%%%%%%%%%%%%

%%%%%%%%%%%%%%%%%%%%%%%%%%%%%%%%%%%%%%%%%%%%%%%%%%%%%%%%%%%%%%%%%%%%%%%%%
%
\subsection{A variant of the Khovanskii-Rolle Theorem.}\label{S:Thm7}

Suppose that we have a system of equations
 \begin{equation}\label{Eq:psisystem}
  \psi_1\ =\ \dotsb\ =\ \psi_{l-1}\ =\ \psi_l\ =\ 0
 \end{equation}
with finitely many solutions in a domain $\Delta\subset\R^l$, and all are non-degenerate. 
Let $C$ be the curve obtained by dropping the last function $\psi_l$
from~\eqref{Eq:psisystem}. 
Let $J$ be the Jacobian determinant of $\psi_1,\dotsc,\psi_l$.\medskip

%%%%%%%%%%%%%%%%%%%%%%%%%%%%%%%%%%%%%%%%%%%
\noindent{\bf Khovanskii-Rolle Theorem.}
{\it
   We have
 \begin{equation}\label{Eq:Kh-Ro}
   \#\calV(\psi_1,\dotsc,\psi_l)\ \leq\ 
   \ubc(C)\ +\ \#\calV(\psi_1,\dotsc,\psi_{l-1},J)\ .
 \end{equation}}
%%%%%%%%%%%%%%%%%%%%%%%%%%%%%%%%%%%%%%%%%%

When the $\psi$ are sums of logarithms of degree 1 polynomials, the Jacobian $J$ is a
polynomial of low degree, after multiplying by the degree 1 polynomials.
This may be iterated as follows.
Drop $\psi_{l-1}$ from the system $\psi_1=\dotsb=\psi_{l-1}=J=0$ to obtain a new curve, and
an inequality of the form~\eqref{Eq:Kh-Ro} involving the unbounded components of
this new curve and a system with two Jacobians which are polynomials of low degree, and so on.

This requires that the successive systems have finitely many solutions, which is
simply not the case, as we have insufficient freedom in the
original system~\eqref{Eq:reduction} to ensure that.
It turns out that the inequality~\eqref{Eq:Kh-Ro} still holds under perturbations of the  
Jacobian, and this is the key to the statement and proof of
Theorem~\ref{T:Modified_Kh_Ro}.\medskip

We compute the multidegree of the numerator of a Jacobian matrix consisting
of partial derivatives of some of the $\phi_m(\bz)$ and of some
polynomials of given multidegrees.
Since $\phi_m(\bz)$ is a linear combination of logarithms of absolute values of the
variables $z_{i,j}$ and the polynomials $q_i(\bz_i)$ the common denominator of
the partial derivatives is
\[
   \Blue{\delta}\ :=\ \prod_{i=1}^n \Bigr( q_i(\bz_i)\cdot \prod_{j=1}^{l_i}
   z_{i,j}\Bigl)\,.
\]
Sine $\delta$ does not vanish on $\calM$, multiplying by $\delta$ will not change any zero set
in $\calM$.

%%%%%%%%%%%%%%%%%%%%%%%%%%%%%%%%%%%%%%%%%%%%%%%%%%%%%%%%%%%%%%%%%%%%%%%%%%%
\begin{theorem}\label{T:det_deg}
 Suppose that for each $m=k{+}1, k{+}2,\dotsc,l$, $F_m(\bz)$ is a polynomial of
 multidegree $d_m$.
 Then the numerator 
\[
   \delta\cdot \det \Jac (\phi_1,\dotsc,\phi_k,\ F_{k+1},\dotsc,F_l)
\]
 of the Jacobian determinant has multidegree $1+d_{k+1}+d_{k+2}+\dotsb+d_l$.
\end{theorem}
%%%%%%%%%%%%%%%%%%%%%%%%%%%%%%%%%%%%%%%%%%%%%%%%%%%%%%%%%%%%%%%%%%%%%%%%%%%

%%%%%%%%%%%%%%%%%%%%%%%%%%%%%%%%%%%%%%%%%%%%%%%%%%%%%%%%%%%%%%%%%%%%%%%%%%%
\begin{proof}
 If we expand the determinant of the Jacobian matrix along its first $k$ rows,
 we obtain a sum of products of $k\times k$ determinants of partial derivatives
 of the logarithms $\phi_m$ by $(l{-}k)\times(l{-}k)$ determinants of partial derivatives
 of the polynomials $F_m(\bz)$. 
 We show that the statement of the theorem holds for each term in this sum.

 A product $\pm\det M_1 \cdot \det M_2$ occurs in this expansion only if 
 $M_1$ is a  $k\times k$ matrix of partial derivatives 
 $\frac{\partial\phi_m}{\partial z_{i,j}}$, $M_2$ a $(l{-}k)\times(l{-}k)$  
 matrix of partial derivatives $\frac{\partial F_m}{\partial z_{i,j}}$,
 and the partial derivatives in $M_1$ are distinct from the partial
 derivatives in $M_2$.
 Thus if $\delta_1$ is the product of all linear polynomials $q_i(\bz_i)$ and of the
 variables occurring as partial derivatives in $M_1$ and
 $\delta_2$ the product of the  variables occurring as partial derivatives in $M_2$,
 then $\delta=\delta_1\cdot\delta_2$ and so
\[
   \delta\bigl(\det M_1 \cdot \det M_2\bigr)\ =\ 
         (\delta_1\det M_1)\cdot (\delta_2\det M_2)\,.
\]

 If we set $M'_2$ to be the matrix obtained from $M_2$ by multiplying each column by the
 corresponding variable, then $\delta_2\det M_2=\det M_2'$.
 A typical entry of $M'_2$ is
\[
    z_{i,j}\frac{\partial F_m(\bz)}{\partial z_{i,j}}\,,
\]
 which is a polynomial of multidegree $d_m$.
 It follows that $\det M'_2$ has multidegree $d_{k{+}1}+d_{k{+}2}+\dotsb+d_l$.
 The theorem now follows from Lemma~\ref{L:Minor} below which shows that 
 $\delta_1\cdot \det M_1$ has multidegree $1$.
\end{proof}
%%%%%%%%%%%%%%%%%%%%%%%%%%%%%%%%%%%%%%%%%%%%%%%%%%%%%%%%%%%%%%%%%%%%%%%%%%% 

Let $M$ be any square submatrix of the Jacobian matrix
$\Blue{\mbox{Jac}}:=( \partial \phi_k/\partial z_{i,j})$.
Since
 \begin{equation}\label{Eq:partialPhi}
   \frac{\partial \phi_k}{\partial z_{i,j}}\ =\ 
    \frac{\alpha^{(k)}_{i,j}}{z_{i,j}}\ +\ \frac{\alpha^{(k)}_{i,0}}{q_i(\bz_i)}\ ,
 \end{equation}
the entries of a submatrix $M$ of $\mbox{Jac}$ will have denominators that include the
variables $z_{i,j}$ corresponding to the columns of $M$, as well as some of the degree 1
polynomials $q_i(\bz_i)$.
Let \Blue{$\delta_M$} be the product of all degree 1 polynomials $q_i(\bz_i)$, together
with all these variables corresponding to columns of $M$.

%%%%%%%%%%%%%%%%%%%%%%%%%%%%%%%%%%%%%%%%%%%%%%%%%%%%
\begin{lemma}\label{L:Minor}
  $\delta_M \det(M)$ is a polynomial with multidegree $1$.
\end{lemma}
%%%%%%%%%%%%%%%%%%%%%%%%%%%%%%%%%%%%%%%%%%%%%%%%%%%%

%%%%%%%%%%%%%%%%%%%%%%%%%%%%%%%%%%%%%%%%%%%%%%%%%%%%
\begin{proof}
 
If no variable in the set $\bz_i$ occurs in $M$, then these variables appear in
$\delta_M\det(M)$ only as the degree 1 polynomial $q_i(\bz_i)$ contained in $\delta_M$.
Suppose now that some variables in $\bz_i$ occur in $M$.  If we expand $\det(M)$ along
the columns corresponding to the variables in $\bz_i$, we obtain a sum of products 
$\det M_i\cdot \det M_i'$ of determinants of submatrices, where $M_i$ only contains
variables from $\bz_i$ and $M'_i$ contains no variables from $\bz_i$.
Hence the statement reduces to the case where only variables in $\bz_i$ occur in
$M$. By~\eqref{Eq:partialPhi}, the columns of $M$ all have the form 
\[
   \frac{v_j}{z_{i,j}}\ +\ \frac{v}{q_i(\bz_i)}\,,
\]
where $v_j$ and $v$ are scalar vectors, and $v$ is the same for all columns.
The determinant is the exterior product of these columns, which we may expand using
multilinearity and antisymmetry. The lemma follows immediately from the form of this expansion, which
we leave to the reader.  
\end{proof}
%%%%%%%%%%%%%%%%%%%%%%%%%%%%%%%%%%%%%%%%%%

%%%%%%%%%%%%%%%%%%%%%%%%%%%%%%%%%%%%%%%%%%%%%%%%%%%%%%%%%%%%%%%%%%%%%%%%%%%%
\begin{proof}[Proof of Theorem~$\ref{T:Modified_Kh_Ro}$]
 We prove both statements by downward induction on $k$, with the first case $k=l$. 
 Observe that (1) holds for $k=l$.
 Suppose that (1) holds for some $k\leq l$.
 Set $\Blue{J}$ to be the numerator of the Jacobian determinant
\[
   \det\Jac(\phi_1,\dotsc,\phi_k,\,F_{k+1},\dotsc,F_l)\,.
\]
 If we set $J_k:=J$, then the usual Khovanskii-Rolle Theorem will imply that statement (2)
 holds, but we would like to ensure that (1) holds for $k{-}1$.

 By Theorem~\ref{T:det_deg}, $J$ has multidegree 
\[
   1+2^{l-k-1}+2^{l-k-2}+\dotsb+2^{l-(l-1)}+2^{l-l}
   \ =\ 2^{l-k}\,.
\]
 By condition (1) $J$ will not vanish at any point of 
 $\calV(\phi_1,\dotsc,\phi_k,\,F_{k+1},\dotsc,F_l)$.
 A general polynomial of multidegree $2^{l-k}$ 
 will intersect the curve $C_k$ as well as the surface 
 $\calV(\phi_1,\dotsc,\phi_{k-2},\,F_{k+1},\dotsc,F_l)$ transversally in $\calM_\C$.
 Let $F_k$ be a polynomial of multidegree $2^{l-k}$ which has the same signs as $J$ at the
 points of  $\calV(\phi_1,\dotsc,\phi_k,\,F_{k+1},\dotsc,F_l)$, but which is also 
 general enough so that (1) holds for $k{-}1$.

 Then (2) holds.
 The reason is the same as for the Khovanskii-Rolle Theorem:
 along any arc of $C_k$ between any two consecutive points where
 $\phi_k$ vanishes, there must be a zero of $J$, as it has different signs at these
 two points.
 But $F_k$ has the same signs at these points as does $J$, so it also must vanish on 
 the arc of $C_k$ between them.
\end{proof}
%%%%%%%%%%%%%%%%%%%%%%%%%%%%%%%%%%%%%%%%%%%%%%%%%%%%%%%%%%%%%%%%%%%%%%%%%%%%

\begin{remark}
 The necessity of this modification of the Khovanskii-Rolle Theorem is that 
 in symbolic computations (done in positive characteristic) when 
 $l_1=l_2=l_3=1$, if we simply set
 \begin{eqnarray*}
   J_3&:=&\delta\det\Jac(\phi_1,\phi_2,\phi_3)\,,\\
   J_2&:=&\delta\det\Jac(\phi_1,\phi_2,J_3)\,,\quad\mbox{and}\\
   J_1&:=&\delta\det\Jac(\phi_1,J_2,J_3)\,,
 \end{eqnarray*}
 then these successive Jacobians do not meet transversally.
 Even worse (for the application of the Khovanskii-Role Theorem), when
 $n=4$ and each $l_i=1$, the computed Jacobians have a common curve of intersection.
\end{remark}

%%%%%%%%%%%%%%%%%%%%%%%%%%%%%%%%%%%%%%%%%%%%%%%%%%%%%%%%%%%%%%%%%%%%%%%%%
%
\subsection{Proof Lemma~\ref{L:first_estimation}}\label{L:Main_est}

For the first statement of Lemma~\ref{L:first_estimation}, in the system 
 \begin{equation}\label{Eq:big_system}
   F_1(\bz)\ =\ F_2(\bz)\ =\ \dotsb =\ F_l(\bz)\ =\ 0\,,
 \end{equation}
 the polynomial $F_k$ has multidegree $2^{l-k}$, by Theorem~\ref{T:Modified_Kh_Ro}.
 Thus the number of non-degenerate real solutions to~\eqref{Eq:big_system} is at most the 
 number of complex solutions to a multilinear system multiplied by
\[
   2^{l-1}\cdot 2^{l-2}\dotsb 2^2 \cdot 2^1 \cdot 2^0\ =\ 
   2^{\binom{l}{2}}\,.
\]
 A multilinear system with blocks of variables 
 $\bz_1,\dotsc,\bz_n$ of respective sizes $l_1,\dotsc,l_n$, has at most
 $\binom{l}{l_1,\dotsc,l_n}$ non-degenerate complex solutions.
 This is a special case of Kuchnirenko's Theorem~\cite{BKK} as the Newton
 polytope of such a multilinear polynomial is the product of unit simplicies of dimensions
 $l_1,\dotsc,l_n$ which has volume $\frac{1}{l_1!}\dotsb\frac{1}{l_n!}$.
 Thus
\[
   2^{\binom{l}{2}}\cdot \tbinom{l}{l_1,\dotsc,l_n}
\]
 is a bound for the number of non-degenerate real solutions to the system~\eqref{Eq:big_system}
 in any domain in $\R\P^{l_1}\times\dotsb\times \R\P^{l_n}$.\medskip

 The second statement is an estimate for the number of unbounded components of the curve
 $C_k$ in either $\calM$ or in some chamber $\Delta$ of $\calM$.
 We first estimate the number of points in either the hypersurface arrangement $\calH$
 (the boundary of $\calM$) or in the boundary of the chamber $\Delta$ that lie in the closure  
 $\overline{C_k}$ of $C_k$.
 We use this to estimate the number of unbounded components of $C_k$.

 Note that $C_k$ is the subset of $\mu_k$ on which
 \begin{equation}\label{eq:Cksystem}
    F_{k+1}(\bz)\ =\ \dotsb\ =\ F_l(\bz)\ =\ 0\,,
 \end{equation}
 holds, so the points of $\overline{C_k}\cap\calH$ are a subset of the points of
 $\overline{\mu_k}\cap\calH$ where~\eqref{eq:Cksystem} holds.

 By Lemma~\ref{L:boundary}, $\overline{\mu_k}\cap\calH$ is a union of $l{-}k$ dimensional
 faces of $\calH$.
 Each such face is the intersection of $k$ of the hypersurfaces in $\calH$ and is therefore
 isomorphic to a product 
 \begin{equation}\label{Eq:RPjprod}
   \R\P^{j_1}\times\dotsb\times\R\P^{j_n}
 \end{equation}
 where $0\leq j_i\leq l_i$ for $i=1,\dotsc,n$ and $j_1+\dotsb+j_n=l-k$.
 By the same arguments we just gave for the first statement, the 
 system~\eqref{eq:Cksystem} has at most 
\[
   2^{\binom{l-k}{2}}\cdot \tbinom{l-k}{j_1,\dotsc,j_n}
\]
 solutions on the face~\eqref{Eq:RPjprod}.

 Each face~\eqref{Eq:RPjprod} is the intersection of exactly $k$ hypersurfaces in $\calH$,
 as these hypersurfaces form a normal crossings divisor.
 Each hypersurface is pulled back from a hyperplane in the arrangement $\calH_i$ in some
 $\R\P^{l_i}$ factor of $\R\P$. 
 If we set $k_i:=l_i-j_i$, then the face~\eqref{Eq:RPjprod} is an
 intersection of $k_i$ hypersurfaces pulled back from $\calH_i$, for $i=1,\dotsc,n$.
 Since $\calH_i$ consists of $l_i+2$ hyperplanes in $\R\P^{l+i}$, there
 are  
\[
   \prod_{i=1}^n \binom{l_i+2}{k_i}\ =\ 
   \prod_{i=1}^n \binom{l_i+2}{j_i+2}
\]
 faces of the form~\eqref{Eq:RPjprod}.
 Thus the number of points of $\overline{C_k}$ lying in $\calH$ is at most 
\[
   2^{\binom{l-k}{2}}\cdot \sum \tbinom{l-k}{j_1,\dotsc,j_n}\cdot
     \prod_{i=1}^n \binom{l_i+2}{j_i+2}\ ,
\]
the sum is over all $j_1,\dotsc,j_n$ with $0\leq j_i\leq l_i$ for $i=1,\dotsc,n$
where $j_1+\dotsb+j_n=l-k$.

Each unbounded component of the curve $C_k$ has two ends which approach points of
$\overline{C_k}\cap\calH$.
We claim that each point of $\overline{C_k}\cap\calH$ has at most $2^k$ branches of $C_k$
approaching it, and thus
\[
   2\cdot \ubc(C_k)\ \leq\ 2^k\cdot 
  2^{\binom{l-k}{2}}\cdot \sum \tbinom{l-k}{j_1,\dotsc,j_n}\cdot
     \prod_{i=1}^n \binom{l_i+2}{j_i+2}\,,
\]
the same sum as before.
This gives the estimate (2) in Lemma~\ref{L:first_estimation}.

To see the claim, note that by Lemma~\ref{L:boundary}, $\mu_k$ has at most $2^k$ branches in
the neighborhood of each point in an $l{-}k$ dimensional face of $\calH$, one for each
incident chamber.
Since $C_k$ consists of the points of $\mu_k$ where~\eqref{eq:Cksystem} holds, the claim
follows as the polynomials in~\eqref{eq:Cksystem} are sufficiently general so
that their common zero set is transverse to any $l{-}k$ face of $\calH$.

%%%%%%%%%%%%%%%%%%%%%%%%%%%%%%%%%%%%%%%%%%%%%%%%%%%%%%%%%%%%%%%%%%%%%%%%%%%%
%
%
\subsection{Proof of Lemma~\ref{L:further estimation} }\label{S:further estimation}
We assume as before that $n>1$.
For $k=0,1,\dotsc,l$, set
\[
   \Blue{a_k}\ :=\ 2^{\binom{l-k}{2}}\cdot \sum \tbinom{l-k}{j_1,\dotsc,j_n}\cdot
     \prod_{i=1}^n \binom{l_i+2}{j_i+2}
\]
the sum over all $j_1,\dotsc,j_n$ with $0\leq j_i\leq l_i$ for $i=1,\dotsc,n$ 
where $j_1+\dotsb+j_n=l-k$.
Note that $a_0=2^{\binom{l}{2}} \tbinom{l}{l_1, \dotsc,l_n}$.
The sum in brackets in~\eqref{big_est} is $\sum_{k=1}^l 2^k \cdot a_k$
and becomes $\sum_{k=1}^l a_k$ for a single chamber $\Delta$.
When $l=3$ and $l=4$, these quantities can be explicitely computed, proving the lemma in
those cases. 
Assume now that $l \geq 5$. We show that
\begin{equation} \label{E:intbound}
a_k \ \leq\  \frac{2^{k-1}}{k!} \cdot a_0 \quad \mbox{for} \quad k=1,\dotsc,l.
\end{equation}
The lemma follows as
\[
  \sum_{k=1}^l a_k \ \leq \
     \Bigl(\sum_{k=1}^l \frac{2^{k-1}}{k!} \Bigr)
      \cdot a_0
    \ <\ \Bigl(\sum_{k=1}^{\infty} \frac{2^{k-1}}{k!} \Bigr)
      \cdot a_0\ =\ \frac{e^2-1}{2}\cdot a_0\,,
\]
and similarly
\[
  \sum_{k=1}^l 2^k a_k \ \leq \
     \Bigl(\sum_{k=1}^l \frac{4^{k-1}}{k!} \Bigr)
      \cdot a_0
    \ < \ \Bigl(\sum_{k=1}^{\infty} \frac{4^{k-1}}{k!} \Bigr)
      \cdot a_0 \ =\ \frac{e^4-1}{2}\cdot a_0\,.\bigskip
\]

For any $k=1,2,\dotsc,l$, we have
 \[
     \tbinom{l}{l_1, \dotsc,l_n}=\sum \tbinom{l-k}{j_1,\dotsc,j_n}\cdot
    \tbinom{k}{l_1-j_1,\dotsc,l_n-j_n}\,,
 \]
 the sum over $j_1,\dotsc,j_n$ with $0\leq j_i\leq l_i$ for $i=1,\dotsc,n$ where $j_1+\dotsb+j_n=l-k$.
To prove~\eqref{E:intbound}, it suffices thus to prove that
 \begin{equation} \label{E:intboundb}
  \prod_{i=1}^n \binom{l_i+2}{j_i+2} \ \leq\ 
  \frac{2^{k-1}}{k !} \cdot 2^{\binom{l}{2}-\binom{l-k}{2}}
  \tbinom{k}{l_1-j_1,\dotsc,l_n-j_n}\,.
 \end{equation}

For this, note that 
 \[
   \prod_{i=1}^n \binom{l_i+2}{j_i+2}\ =\ 
   \frac{1}{k !} \cdot \tbinom{k}{l_1-j_1,\dotsc,l_n-j_n} 
   \cdot \prod_{i=1}^n \frac{(l_i+2)!}{(j_i+2)!}\,.
 \]
Then observe that 
 \begin{equation} \label{E:intboundbis}
  \prod_{i=1}^n \frac{(l_i+2)!}{(j_i+2)!}\ =\ 
  \prod_{i=1}^n (l_i{+}2)(l_i{+}1)\dotsb(j_i+3) \ \leq\ 
  \prod_{i=1}^n (l+1)^{l_i-j_i}\ =\ (l+1)^k\,,
 \end{equation}
 as $\sum_i(l_i-j_i)=k$ and we have $l_i+2<l+1$ since
 $l=l_1+\dotsb+l_n$ with each $l_i>0$ and we assumed that $n>1$.

 Now, as $l\geq 5$, we have $l+1<2^{\frac{l+1}{2}-\frac{1}{l}}$ but we also have
\[
  l-\frac{x-1}{2}-\frac{1}{x}\ \geq \ l-\frac{l-1}{2}-\frac{1}{l}\ =\ 
   \frac{l+1}{2}-\frac{1}{l}\,,
\]
for $1\leq x\leq l$.
Thus 
 \begin{equation}\label{eq:powers}
    (l+1)^k\ \leq\ \bigl(2^{\frac{l+1}{2}-\frac{1}{l}}\bigr)^k
    \leq\ \bigl(2^{l-\frac{k-1}{2}-\frac{1}{k}}\bigr)^k\ =\ 2^{k-1+\binom{l}{2}-\binom{l-k}{2}}\,.    
 \end{equation}
Putting~\eqref{E:intboundbis} together with~\eqref{eq:powers}
establishes~\eqref{E:intboundb}, and 
completes the proof of Lemma~\ref{L:further estimation}.

%%%%%%%%%%%%%%%%%%%%%%%%%%%%%%%%%%%%%%%%%%%%%%%%%%%%%%%%%%%%%%%%%%%%%%%%%%%%
%
%
\section*{Acknowledgments}
We thank the Centre Interfaculaire Bernoulli at the EPFL in Lausanne,
Switzerland, where we began this project.

%%%%%%%%%%%%%%%%%%%%%%%%%%%%%%%%%%%%%%%%%%%%%%%%%%%%%%%%%%%%%%%%%%%%%%%%%%%%
\providecommand{\bysame}{\leavevmode\hbox to3em{\hrulefill}\thinspace}
\providecommand{\MR}{\relax\ifhmode\unskip\space\fi MR }
\providecommand{\MRhref}[2]{%
  \href{http://www.ams.org/mathscinet-getitem?mr=#1}{#2}
}
\providecommand{\href}[2]{#2}

%%%%%%%%%%%%%%%%%%%%%%%%%%%%%%%%%%%%%%%%%%%%%%%%%%%%%%%%%%%%%%%%%%%%%%%%%%%%
\end{document}